\documentclass[11pt]{article}\UseRawInputEncoding
\usepackage{amssymb,amsfonts,amsmath,latexsym,epsf,tikz,url}
\usepackage[usenames,dvipsnames]{pstricks}
\usepackage{pstricks-add}
\usepackage{epsfig}
\usepackage{pst-grad} 
\usepackage{pst-plot} 
\usepackage[space]{grffile} 
\usepackage{etoolbox} 
\makeatletter 
\patchcmd\Gread@eps{\@inputcheck#1 }{\@inputcheck"#1"\relax}{}{}
\makeatother

\newtheorem{theorem}{Theorem}[section]
\newtheorem{proposition}[theorem]{Proposition}

\newtheorem{corollary}[theorem]{Corollary}

\newtheorem{remark}[theorem]{Remark}

\newcommand{\qed}{\hfill $\square$\medskip}

\begin{document}

\title{Some results on the super domination number of a graph II}

\author{
Nima Ghanbari
}

\date{\today}

\maketitle

\begin{center}
Department of Informatics, University of Bergen, P.O. Box 7803, 5020 Bergen, Norway\\
{\tt   Nima.Ghanbari@uib.no}
\end{center}


\begin{abstract}
Let $G=(V,E)$ be a simple graph. A dominating set of $G$ is a subset $S\subseteq V$ such that every vertex not in $S$ is adjacent to at least one vertex in $S$.
The cardinality of a smallest dominating set of $G$, denoted by $\gamma(G)$, is the domination number of $G$. A dominating set $S$ is called a super dominating set of $G$, if for every vertex  $u\in \overline{S}=V-S$, there exists $v\in S$ such that $N(v)\cap \overline{S}=\{u\}$. The cardinality of a smallest super dominating set of $G$, denoted by $\gamma_{sp}(G)$, is the super domination number of $G$. In this paper, we obtain more results on the super domination number of graphs which is modified by an operation on vertices. Also, we present some sharp bounds for super domination number of chain and bouquet of pairwise disjoint connected graphs. 
\end{abstract}

\noindent{\bf Keywords:} domination number, super dominating set, chain, bouquet

\medskip
\noindent{\bf AMS Subj.\ Class.:} 05C69, 05C76

\section{Introduction}
 
Let $G = (V,E)$ be a  graph with vertex set $V$ and edge set $E$. Throughout this paper, we consider graphs without loops and directed edges.
For each vertex $v\in V$, the set $N(v)=\{u\in V | uv \in E\}$ refers to the open neighbourhood of $v$ and the set $N[v]=N(v)\cup \{v\}$ refers to the closed neighbourhood of $v$ in $G$. The degree of $v$, denoted by $\deg(v)$, is the cardinality of $N(v)$.
 A set $S\subseteq V$ is a  dominating set if every vertex in $\overline{S}= V- S$ is adjacent to at least one vertex in $S$.
The  domination number $\gamma(G)$ is the minimum cardinality of a dominating set in $G$. There are various domination numbers in the literature.
For a detailed treatment of domination theory, the reader is referred to \cite{domination}. 

\medskip

The concept of super domination number was introduced by Lema\'nska et al. in 2015 \cite{Lemans}. A dominating set $S$ is called a super dominating set of $G$, if for every vertex  $u\in \overline{S}$, there exists $v\in S$ such that $N(v)\cap \overline{S}=\{u\}$. The cardinality of a smallest super dominating set of $G$, denoted by $\gamma_{sp}(G)$, is the super domination number of $G$. We refer the reader to \cite{Alf,Dett,Nima,Kri,Kle,Zhu} for more details on super dominating set of a graph. 

\medskip

Let $G$ be a connected graph constructed from pairwise disjoint connected graphs
$G_1,\ldots ,G_n$ as follows. Select a vertex of $G_1$, a vertex of $G_2$, and identify these two vertices. Then continue in this manner inductively.  Note that the graph $G$ constructed in this way has a tree-like structure, the $G_i$'s being its building stones (see Figure \ref{Figure1}).

\begin{figure}[!h]
	\begin{center}
		\psscalebox{0.6 0.6}
{
\begin{pspicture}(0,-4.819607)(13.664668,2.90118)
\pscircle[linecolor=black, linewidth=0.04, dimen=outer](5.0985146,1.0603933){1.6}
\pscustom[linecolor=black, linewidth=0.04]
{
\newpath
\moveto(11.898515,0.66039336)
}
\pscustom[linecolor=black, linewidth=0.04]
{
\newpath
\moveto(11.898515,0.26039338)
}
\pscustom[linecolor=black, linewidth=0.04]
{
\newpath
\moveto(12.698514,0.66039336)
}
\pscustom[linecolor=black, linewidth=0.04]
{
\newpath
\moveto(10.298514,1.0603933)
}
\pscustom[linecolor=black, linewidth=0.04]
{
\newpath
\moveto(11.098515,-0.9396066)
}
\pscustom[linecolor=black, linewidth=0.04]
{
\newpath
\moveto(11.098515,-0.9396066)
}
\pscustom[linecolor=black, linewidth=0.04]
{
\newpath
\moveto(11.898515,0.66039336)
}
\pscustom[linecolor=black, linewidth=0.04]
{
\newpath
\moveto(11.898515,-0.9396066)
}
\pscustom[linecolor=black, linewidth=0.04]
{
\newpath
\moveto(11.898515,-0.9396066)
}
\pscustom[linecolor=black, linewidth=0.04]
{
\newpath
\moveto(12.698514,-0.9396066)
}
\pscustom[linecolor=black, linewidth=0.04]
{
\newpath
\moveto(12.698514,0.26039338)
}
\pscustom[linecolor=black, linewidth=0.04]
{
\newpath
\moveto(14.298514,0.66039336)
\closepath}
\psbezier[linecolor=black, linewidth=0.04](11.598515,1.0203934)(12.220886,1.467607)(12.593457,1.262929)(13.268515,1.0203933715820312)(13.943572,0.7778577)(12.308265,0.90039337)(12.224765,0.10039337)(12.141264,-0.69960666)(10.976142,0.5731798)(11.598515,1.0203934)
\psbezier[linecolor=black, linewidth=0.04](4.8362556,-3.2521083)(4.063277,-2.2959895)(4.6714916,-1.9655427)(4.891483,-0.99004078729821)(5.111474,-0.014538889)(5.3979383,-0.84551746)(5.373531,-1.8452196)(5.349124,-2.8449216)(5.6092343,-4.208227)(4.8362556,-3.2521083)
\psbezier[linecolor=black, linewidth=0.04](8.198514,-2.0396066)(6.8114076,-1.3924998)(6.844908,-0.93520766)(5.8785143,-1.6996066284179687)(4.9121203,-2.4640057)(5.6385145,-3.4996066)(6.3385143,-2.8396065)(7.0385146,-2.1796067)(9.585621,-2.6867135)(8.198514,-2.0396066)
\pscircle[linecolor=black, linewidth=0.04, dimen=outer](7.5785146,-3.6396067){1.18}
\psdots[linecolor=black, dotsize=0.2](11.418514,0.7403934)
\psdots[linecolor=black, dotsize=0.2](9.618514,1.5003934)
\psdots[linecolor=black, dotsize=0.2](6.6585145,0.7403934)
\psdots[linecolor=black, dotsize=0.2](3.5185144,0.96039337)
\psdots[linecolor=black, dotsize=0.2](5.1185145,-0.51960665)
\psdots[linecolor=black, dotsize=0.2](5.3985143,-2.5796065)
\psdots[linecolor=black, dotsize=0.2](7.458514,-2.4596066)
\rput[bl](8.878514,0.42039338){$G_i$}
\rput[bl](7.478514,-4.1196065){$G_j$}
\psbezier[linecolor=black, linewidth=0.04](0.1985144,0.22039337)(0.93261385,0.89943534)(2.1385605,0.6900083)(3.0785143,0.9403933715820313)(4.0184684,1.1907784)(3.248657,0.442929)(2.2785144,0.20039338)(1.3083719,-0.042142253)(-0.53558505,-0.45864862)(0.1985144,0.22039337)
\psbezier[linecolor=black, linewidth=0.04](2.885918,1.4892112)(1.7389486,2.4304078)(-0.48852357,3.5744174)(0.5524718,2.1502930326916756)(1.5934672,0.7261687)(1.5427756,1.2830372)(2.5062277,1.2429687)(3.46968,1.2029002)(4.0328875,0.5480146)(2.885918,1.4892112)
\psellipse[linecolor=black, linewidth=0.04, dimen=outer](9.038514,0.7403934)(2.4,0.8)
\psbezier[linecolor=black, linewidth=0.04](9.399693,1.883719)(9.770389,2.812473)(12.016343,2.7533927)(13.011008,2.856550531577144)(14.005673,2.9597082)(13.727474,2.4925284)(12.761896,2.2324166)(11.796317,1.9723049)(9.028996,0.9549648)(9.399693,1.883719)
\pscircle[linecolor=black, linewidth=0.04, dimen=outer](9.898515,-3.3396065){1.2}
\psellipse[linecolor=black, linewidth=0.04, dimen=outer](2.2985144,-1.1396066)(0.4,1.4)
\psellipse[linecolor=black, linewidth=0.04, dimen=outer](2.4985144,-3.3396065)(1.8,0.8)
\psdots[linecolor=black, dotsize=0.2](2.2985144,0.26039338)
\psdots[linecolor=black, dotsize=0.2](2.2985144,-2.5396066)
\psdots[linecolor=black, dotsize=0.2](8.698514,-3.3396065)
\psdots[linecolor=black, dotsize=0.2](9.898515,-2.1396067)
\psellipse[linecolor=black, linewidth=0.04, dimen=outer](10.298514,-1.5396066)(2.0,0.6)
\end{pspicture}
}
	\end{center}
	\caption{\label{Figure1} A graph with subgraph units  $G_1,\ldots , G_n$.}
\end{figure}
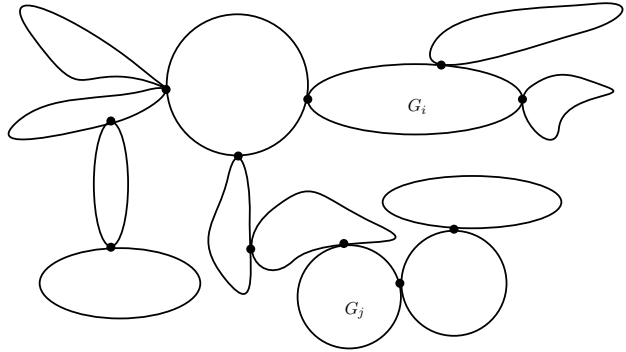

Usually  say that $G$ is obtained by point-attaching from $G_1,\ldots , G_n$ and that $G_i$'s are the primary subgraphs of $G$. A particular case of this construction is the decomposition of a connected graph into blocks (see \cite{Deutsch}).
We refer the reader to \cite{Alikhani1,Nima0,Moster} for more details and results in the concept of graphs from primary subgraphs.

\medskip

In this paper, we continue the study of super domination number of a graph. In Section 2, we  mention some previous results, also the definition of $G\odot v$ and find a shrap upper bound for the super domination number of that. In Section 3, we obtain some results on the chain of graphs that is a special case of graphs which are obtained by point-attaching from primary subgraphs. Finally, in Section 4, we find some sharp bounds  on the super domination number of the bouquet of graphs  which is another special case of graphs that are made by point-attaching.
    
\section{Super domination number of $G\odot v$}
$G\odot v$ is the graph obtained from $G$ by the removal of all edges between
 any pair of neighbours of $v$ \cite{Alikhani}. Some results in this operation can be found in \cite{Nima1}. In this section, we study the super domination number of $G\odot v$.
First, we state some known results.

	\begin{theorem}\cite{Lemans}\label{thm-1}
 Let $G$ be a graph of order $n$ which is not empty graph. Then,
 $$1\leq \gamma(G) \leq \frac{n}{2} \leq \gamma_{sp}(G) \leq n-1.$$
	\end{theorem}

	\begin{theorem}\cite{Lemans}\label{thm-2}
	\begin{itemize}
	\item[(i)]
	For a path graph $P_n$ with $n\geq 3$, $\gamma_{sp}(P_n)=\lceil \frac{n}{2} \rceil$.
	\item[(ii)]
	For a cycle graph $C_n$,
	\begin{displaymath}
\gamma_{sp}(C_n)= \left\{ \begin{array}{ll}
\lceil\frac{n}{2}\rceil & \textrm{if $n \equiv 0, 3 ~ (mod ~ 4)$, }\\
\\
\lceil\frac{n+1}{2}\rceil & \textrm{otherwise.}
\end{array} \right.
\end{displaymath}
\item[(iii)]
For the complete graph $K_n$, $\gamma_{sp}(K_n)=n-1$.
\item[(iv)]
For the complete bipartite graph $K_{n,m}$, $\gamma_{sp}(K_{n,m})=n+m-2$, where $min\{n,m\}\geq 2$.
\item[(v)]
For the star graph $K_{1,n}$, $\gamma_{sp}(K_{1,n})=n$.
\end{itemize}
	\end{theorem}

	\begin{theorem}\cite{Nima}\label{Firend-thm}
 For the friendship graph $F_n$, $\gamma_{sp}(F_n)=n+1$.
	\end{theorem}

		\begin{theorem}\cite{Nima}\label{G/v}
 Let $G=(V,E)$ be a graph and $v\in V$ is not a pendant vertex.  Then,
 $$ \gamma_{sp}(G/v)\leq \gamma_{sp} (G)+\lfloor \frac{\deg (v)}{2} \rfloor -1,$$
 where $G/v$ is the graph obtained by deleting $v$ and putting a clique on the open neighbourhood of $v$.
	\end{theorem}

Here we consider to $G\odot v$. First suppose that $v$ is a pendant vertex. Then by the definition of $G\odot v$, we have $G\odot v=G$. So we have the following easy result:

	\begin{proposition}\label{Godotvpendant}
  Let $G=(V,E)$ be a graph and $v\in V$ is a pendant vertex. Then,
 $$ \gamma_{sp}(G\odot v)= \gamma_{sp} (G).$$
	\end{proposition}

Hence, there is no reason to compute $\gamma_{sp}(G\odot v)$ when $v$ is a pendant vertex.	Now we find a sharp upper bound for the super domination number of $G\odot v$ when $v$ is not a pendant vertex.

	\begin{theorem}\label{Godotv}
 Let $G=(V,E)$ be a graph and $v\in V$ is not a pendant vertex.  Then,
 $$\gamma_{sp}(G\odot v)\leq \gamma_{sp} (G)+\lfloor \frac{\deg (v)}{2} \rfloor -1.$$
	\end{theorem}

	\begin{proof}
Suppose that $v\in V$ such that  $\deg (v)=n\geq2$ and $N(v)=\{v_1,v_2,\ldots,v_n\}$.  Also $D$ is a super dominating set for $G$. We have the following cases:
\begin{itemize}
\item[(i)]
$v\notin D$. So, there exists $v_r\in N(v)\cap D$ such that $N(v_r)\cap \overline{D} = \{v\}$ which means that all other neighbours of $v_r$ are in $D$ too. There is no vertex such as $v_p\in N(v)$  that dominates $v_q\in N(v)\cap \overline{D}$ and satisfies the condition of super dominating set, because in that case we have  $\{v_q,v\}\subseteq N(v_p)\cap \overline{D} $ which is a contradiction. So all vertices in $ N(v)\cap \overline{D}$ dominate by some vertices which are not in $N(v)$.  Now by removing all edges between
 any pair of neighbours of $v$, $D$ is a super dominating set for the $G\odot v$ too. So, $\gamma_{sp}(G\odot v)\leq \gamma_{sp} (G)$.
\item[(ii)]
$v\in D$ and for some $1 \leq i \leq n$, there exists $v_i\in N(v)$ such that  $N(v)\cap \overline{D} = \{v_i\}$. So, all other neighbours of $v$ should be in $D$. Now by removing all edges between
 any pair of neighbours of $v$, $D$ is a super dominating set for the $G\odot v$ too. So, $\gamma_{sp}(G\odot v)\leq \gamma_{sp} (G)$.
\item[(iii)]
$v\in D$ and for every $1 \leq i \leq n$, there does not exist $v_i\in N(v)$ such that  $N(v)\cap \overline{D} = \{v_i\}$. If $v_i\in N(v)$ is dominated by $v_i'$ such that $v_i'\notin N(v)$, then after removing all edges between
 any pair of neighbours of $v$, $v_i$ still can dominated by $v_i'$ and $N(v_i')\cap \overline{D} = \{v_i\}$. So we keep all vertices in $D-N(v)$ in our dominating set. If $v_i\in N(v)$ is dominated by $v_j$ such that $v_j\in N(v)$ and $N(v_j)\cap \overline{D} = \{v_i\}$, then we simply add $v_i$ to our dominating set after removing all edges between
 any pair of neighbours of $v$. At most we have $\lfloor \frac{n}{2} \rfloor$ vertices with this condition. Without loss of generality, suppose that $v_1$ dominates $v_2$, $v_3$ dominates $v_4$, $v_5$ dominates $v_6$ and so on. Since all vertices in $N(v)-\{v_2\}$ are in $D\cup\{v_4,v_6,\ldots\}$, then $v_2$ is now dominated by $v$, and   then by our argument, $D\cup\{v_4,v_6,\ldots\}$ is a super dominating set for $G\odot v$. Hence, $\gamma_{sp}(G\odot v)\leq \gamma_{sp} (G)+\lfloor \frac{n}{2} \rfloor -1$. 
\end{itemize}
  Therefore we have the result.
\qed
	\end{proof}

	\begin{remark}
Upper bound in Theorem \ref{Godotv} is sharp. It suffices to consider $G=F_n$ as friendship graph and $v$ the vertex with $\deg(v)=2n$. By Theorem \ref{Firend-thm}, $\gamma_{sp} (G)=n+1$. One can easily check that $G\odot v=K_{1,2n}$ and then by Theorem \ref{thm-2}, $\gamma_{sp}(G\odot v)=2n$. Therefore $\gamma_{sp}(G\odot v)= \gamma_{sp} (G)+\lfloor \frac{\deg (v)}{2} \rfloor -1$.
	\end{remark}

We end this section by an immediate result of Theorems \ref{G/v} and \ref{Godotv}.

	\begin{corollary}
Let $G=(V,E)$ be a graph and $v\in V$ is not a pendant vertex. Then,
$$ \gamma _{sp}(G) \geq \frac{\gamma _{sp}(G\odot v)+\gamma _{sp}(G/v)}{2}-\lfloor \frac{\deg (v)}{2} \rfloor +1.$$
	\end{corollary}

\section{Super domination number of chain of graphs}

In this section, we consider to a special case of graphs which is obtained by point-attaching from primary subgraphs, and  is called chain of graphs $G_1,\ldots , G_n$. Suppose that $x_i,y_i \in V(G_i)$. Let $C(G_1,...,G_n)$ be the chain of graphs $\{G_i\}_{i=1}^n$ with respect to the vertices $\{x_i, y_i\}_{i=1}^k$ which is obtained by identifying the vertex $y_i$ with the vertex $x_{i+1}$ for $i=1,2,\ldots,n-1$ (see Figure \ref{chain-n}).

	\begin{figure}[!h]
		\begin{center}
			\psscalebox{0.75 0.75}
			{
				\begin{pspicture}(0,-3.9483333)(12.236668,-2.8316667)
				\psellipse[linecolor=black, linewidth=0.04, dimen=outer](1.2533334,-3.4416668)(1.0,0.4)
				\psellipse[linecolor=black, linewidth=0.04, dimen=outer](3.2533333,-3.4416668)(1.0,0.4)
				\psellipse[linecolor=black, linewidth=0.04, dimen=outer](5.2533336,-3.4416668)(1.0,0.4)
				\psellipse[linecolor=black, linewidth=0.04, dimen=outer](8.853333,-3.4416668)(1.0,0.4)
				\psellipse[linecolor=black, linewidth=0.04, dimen=outer](10.853333,-3.4416668)(1.0,0.4)
				\psdots[linecolor=black, fillstyle=solid, dotstyle=o, dotsize=0.3, fillcolor=white](2.2533333,-3.4416666)
				\psdots[linecolor=black, fillstyle=solid, dotstyle=o, dotsize=0.3, fillcolor=white](0.25333345,-3.4416666)
				\psdots[linecolor=black, fillstyle=solid, dotstyle=o, dotsize=0.3, fillcolor=white](2.2533333,-3.4416666)
				\psdots[linecolor=black, fillstyle=solid, dotstyle=o, dotsize=0.3, fillcolor=white](4.2533336,-3.4416666)
				\psdots[linecolor=black, fillstyle=solid, dotstyle=o, dotsize=0.3, fillcolor=white](4.2533336,-3.4416666)
				\psdots[linecolor=black, fillstyle=solid, dotstyle=o, dotsize=0.3, fillcolor=white](9.853333,-3.4416666)
				\psdots[linecolor=black, fillstyle=solid, dotstyle=o, dotsize=0.3, fillcolor=white](9.853333,-3.4416666)
				\psdots[linecolor=black, fillstyle=solid, dotstyle=o, dotsize=0.3, fillcolor=white](11.853333,-3.4416666)
				\rput[bl](0.0,-3.135){$x_1$}
				\rput[bl](2.0400002,-3.2016668){$x_2$}
				\rput[bl](3.9866667,-3.1216667){$x_3$}
				\rput[bl](2.1733334,-3.9483335){$y_1$}
				\rput[bl](4.12,-3.9483335){$y_2$}
				\rput[bl](6.1733336,-3.8816667){$y_3$}
				\rput[bl](0.9600001,-3.6283333){$G_1$}
				\rput[bl](3.0,-3.5883334){$G_2$}
				\rput[bl](5.04,-3.5616667){$G_3$}
				\psdots[linecolor=black, fillstyle=solid, dotstyle=o, dotsize=0.3, fillcolor=white](6.2533336,-3.4416666)
				\psdots[linecolor=black, fillstyle=solid, dotstyle=o, dotsize=0.3, fillcolor=white](7.8533335,-3.4416666)
				\psdots[linecolor=black, dotsize=0.1](6.6533337,-3.4416666)
				\psdots[linecolor=black, dotsize=0.1](7.0533333,-3.4416666)
				\psdots[linecolor=black, dotsize=0.1](7.4533334,-3.4416666)
				\rput[bl](9.6,-3.0816667){$x_n$}
				\rput[bl](11.826667,-3.8683333){$y_n$}
				\rput[bl](9.586667,-3.9483335){$y_{n-1}$}
				\rput[bl](8.533334,-3.6016667){$G_{n-1}$}
				\rput[bl](7.4,-3.1616666){$x_{n-1}$}
				\rput[bl](10.613334,-3.575){$G_n$}
				\end{pspicture}
			}
		\end{center}
		\caption{Chain of $n$ graphs $G_1,G_2, \ldots , G_n$.} \label{chain-n}
	\end{figure}
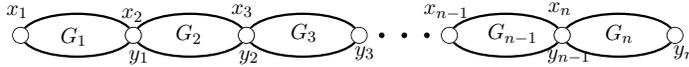

Before we start the study of super domination number of chain of graphs,  we mention the  following easy result which is a direct result of the definition of super dominating set and super domination number:

\begin{proposition}\label{pro-disconnect}
	Let $G$ be a disconnected graph with components $G_1$ and $G_2$. Then
		$$\gamma _{sp}(G)=\gamma _{sp}(G_1)+\gamma _{sp}(G_2).$$
\end{proposition}

Now we consider to chain of two graphs and find sharp upper and lower bounds for its super domination number.

	\begin{theorem}\label{chain2-thm}
	Let $G_1$ and $G_2$ be two  disjoint connected graphs and let
		$x_i,y_i \in V(G_i)$ for $i\in \{1,2\}$. Let $C(G_1,G_2)$ be the chain of graphs $\{G_i\}_{i=1}^2$ with respect to the vertices $\{x_i, y_i\}_{i=1}^2$ which obtained by identifying the vertex $y_1$ with the vertex $x_{2}$. Let this vertex in $V(C(G_1,G_2))$ be $z$ (see Figure \ref{chain-2}). Then,
$$\gamma _{sp}(G_1)+\gamma _{sp}(G_2) -1\leq   \gamma _{sp}(C(G_1,G_2)) \leq \gamma _{sp}(G_1)+\gamma _{sp}(G_2).$$
	\end{theorem}

	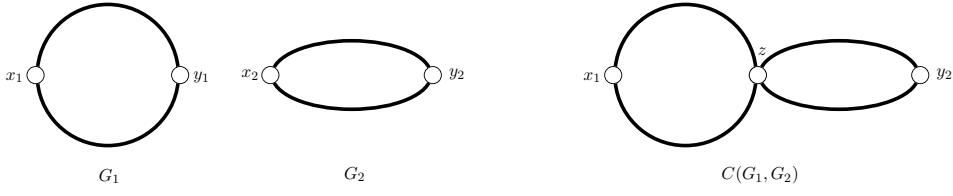
\begin{figure}
		\begin{center}
			\psscalebox{0.6 0.6}
{
\begin{pspicture}(0,-4.8)(20.99,-0.8)
\pscircle[linecolor=black, linewidth=0.08, dimen=outer](2.26,-2.4){1.6}
\psellipse[linecolor=black, linewidth=0.08, dimen=outer](7.66,-2.4)(1.8,0.8)
\psdots[linecolor=black, dotstyle=o, dotsize=0.4, fillcolor=white](0.66,-2.4)
\psdots[linecolor=black, dotstyle=o, dotsize=0.4, fillcolor=white](3.86,-2.4)
\psdots[linecolor=black, dotstyle=o, dotsize=0.4, fillcolor=white](5.86,-2.4)
\psdots[linecolor=black, dotstyle=o, dotsize=0.4, fillcolor=white](9.46,-2.4)
\rput[bl](0.0,-2.56){$x_1$}
\rput[bl](5.22,-2.5){$x_2$}
\rput[bl](4.16,-2.56){$y_1$}
\rput[bl](9.82,-2.5){$y_2$}
\rput[bl](2.06,-4.8){$G_1$}
\rput[bl](7.5,-4.76){$G_2$}
\pscircle[linecolor=black, linewidth=0.08, dimen=outer](15.06,-2.4){1.6}
\psellipse[linecolor=black, linewidth=0.08, dimen=outer](18.46,-2.4)(1.8,0.8)
\psdots[linecolor=black, dotstyle=o, dotsize=0.4, fillcolor=white](13.46,-2.4)
\psdots[linecolor=black, dotstyle=o, dotsize=0.4, fillcolor=white](16.66,-2.4)
\psdots[linecolor=black, dotstyle=o, dotsize=0.4, fillcolor=white](16.66,-2.4)
\psdots[linecolor=black, dotstyle=o, dotsize=0.4, fillcolor=white](20.26,-2.4)
\rput[bl](12.8,-2.56){$x_1$}
\rput[bl](20.62,-2.5){$y_2$}
\rput[bl](15.86,-4.8){$C(G_1,G_2)$}
\rput[bl](16.66,-1.96){$z$}
\end{pspicture}
}
		\end{center}
		\caption{Graphs $G_1,G_2$ and $C(G_1,G_2)$ with respect to vertices $y_1$ and $x_2$, respectively.} \label{chain-2}
	\end{figure}	
	
	\begin{proof}
First, we find an upper bound for $\gamma _{sp}(C(G_1,G_2))$. Let $S_1$ be a super dominating set of $G_1$ and $\gamma _{sp}(G_1)=|S_1|$, and also  $S_2$ be a super dominating set of $G_2$ and $\gamma _{sp}(G_2)=|S_2|$. we have the following cases: 
\begin{itemize}
\item[(i)]
$y_1 \in S_1$ and $x_2 \in S_2$. In this case, $y_1$ and $x_2$ may have or have not influence on the vertices in $\overline{S_1}$ and $\overline{S_2}$, respectively. So we consider the following cases:
\begin{itemize}
\item[(i.1)]
There exists $g_1\in N(y_1)$ such that  $N(y_1)\cap \overline{S_1} = \{g_1\}$ and $g_2\in N(x_2)$ such that  $N(x_2)\cap \overline{S_2} = \{g_2\}$. So $N(y_1)-\{g_1\}\subseteq S_1$  and $N(y_2)-\{g_2\}\subseteq S_2$.  Let
$$S=\left( S_1 \cup S_2 \cup \{z,g_1\} \right)-\{y_1,x_2\}. $$
$S$ is a super dominating set for $C(G_1,G_2)$, because $g_2$ is dominated by $z$ and since all neighbours of $y_1$ are in $S$ now, then $N(z)\cap \overline{S} = \{g_2\}$. The rest of vertices in $\overline{S}$ are dominated by the same vertex as before and the definition of super dominating set holds. So in this case,
$$ \gamma _{sp}(C(G_1,G_2)) \leq \gamma _{sp}(G_1)+\gamma _{sp}(G_2).$$

\item[(i.2)]
There exists $g_1\in N(y_1)$ such that  $N(y_1)\cap \overline{S_1} = \{g_1\}$ but there does not exist $g_2\in N(x_2)$ such that  $N(x_2)\cap \overline{S_2} = \{g_2\}$. We know that $N(y_1)-\{g_1\}\subseteq S_1$, but  we may have more than one vertex in $N(x_2)\cap \overline{S_2}$ or may have all vertices in $N(x_2)$ as a subset of $S_2$. Since we have no knowledge about $N(x_2)\cap \overline{S_2}$, let
$$S=\left( S_1 \cup S_2 \cup \{z,g_1\} \right)-\{y_1,x_2\}. $$
Clearly, $S$ is a super dominating set for $C(G_1,G_2)$ since all vertices in $\overline{S}$ are dominated by the same vertex as before. Hence
$$ \gamma _{sp}(C(G_1,G_2)) \leq \gamma _{sp}(G_1)+\gamma _{sp}(G_2).$$

\item[(i.3)]
There does not exist $g_1\in N(y_1)$ such that  $N(y_1)\cap \overline{S_1} = \{g_1\}$ but there exists $g_2\in N(x_2)$ such that  $N(x_2)\cap \overline{S_2} = \{g_2\}$. It is similar to part (i.2).

\item[(i.4)]
There does not exist $g_1\in N(y_1)$ such that  $N(y_1)\cap \overline{S_1} = \{g_1\}$ and there does not exist $g_2\in N(x_2)$ such that  $N(x_2)\cap \overline{S_2} = \{g_2\}$. We may have more than one vertex in $N(y_1)\cap \overline{S_1}$ or may have all vertices in $N(y_1)$ as a subset of $S_1$, and  same argument about $x_2$. Let
$$S=\left( S_1 \cup S_2 \cup \{z\} \right)-\{y_1,x_2\}. $$
Then all vertices in $\overline{S}$ are dominated by the same vertex as before and the definition of the super dominating set holds. So we have 
$$ \gamma _{sp}(C(G_1,G_2)) \leq \gamma _{sp}(G_1)+\gamma _{sp}(G_2)-1.$$

\end{itemize}

\item[(ii)]
$y_1 \in S_1$ and $x_2 \notin S_2$. In this case, we only pay attention to  $y_1$. So we consider the following cases:

\begin{itemize}
\item[(ii.1)]
There exists $g_1\in N(y_1)$ such that  $N(y_1)\cap \overline{S_1} = \{g_1\}$. So $N(y_1)-\{g_1\}\subseteq S_1$.  Let
$$S=\left( S_1 \cup S_2 \cup \{g_1\} \right)-\{y_1\}. $$
We show that $S$ is a super dominating set for $C(G_1,G_2)$. By the definition of $S$ we have $g_1\in S$, so we do not need to consider it in the definition of super dominating set. Since $x_2 \notin S_2$, then there exists $h\in S_2$ such that  $N(h)\cap \overline{S_2} = \{x_2\}$. Now we consider to $z$ and clearly we have $N(h)\cap \overline{S} = \{z\}$. The rest of vertices in $\overline{S}$ are dominated by the same vertex as before and the definition of the super dominating set holds. So 
$$ \gamma _{sp}(C(G_1,G_2)) \leq \gamma _{sp}(G_1)+\gamma _{sp}(G_2).$$

\item[(ii.2)]
There does not exist $g_1\in N(y_1)$ such that  $N(y_1)\cap \overline{S_1} = \{g_1\}$. So simply let 
$$S=\left( S_1 \cup S_2 \cup \{z\} \right)-\{y_1\}. $$
By an easy argument same as before, we conclude that $S$ is a super dominating set for $C(G_1,G_2)$ and therefore
$$ \gamma _{sp}(C(G_1,G_2)) \leq \gamma _{sp}(G_1)+\gamma _{sp}(G_2).$$

\end{itemize}

\item[(iii)]
$y_1 \notin S_1$ and $x_2 \in S_2$. It is similar to part (ii).

\item[(iv)]
$y_1 \notin S_1$ and $x_2 \notin S_2$. Let $S=S_1 \cup S_2$. Then by similar argument as part (ii.1), $S$ is a super dominating set for $C(G_1,G_2)$ and hence
$$ \gamma _{sp}(C(G_1,G_2)) \leq \gamma _{sp}(G_1)+\gamma _{sp}(G_2).$$

\end{itemize}
Therefore in all cases we have $\gamma _{sp}(C(G_1,G_2)) \leq \gamma _{sp}(G_1)+\gamma _{sp}(G_2)$. Now we find a lower bound for $\gamma _{sp}(C(G_1,G_2))$. First we find a super dominating set for $C(G_1,G_2)$. Let this set be $D$ and $\gamma _{sp}(C(G_1,G_2))=|D|$. Now by using this set, we find super dominating sets for $G_1$ and $G_2$. Consider to the following cases:

\begin{itemize}
\item[(i)]
$z\in D$. In this case, $z$  may has or has not influence on the vertices in $\overline{D}$. So we consider the following cases:
\begin{itemize}
\item[(i.1)]
There exists $u\in N(z)$ such that  $N(z)\cap \overline{D} = \{u\}$. So $N(z)-\{u\}\subseteq D$ and therefore all other neighbours of $z$ are in $D$. Without loss of generality, suppose that $u\in V(G_1)$. Now we separate components $G_1$ and $G_2$ from $C(G_1,G_2)$ and form a disconnected graph with components $G_1$ and $G_2$, replace vertex $z$ with $x_2$ in $G_1$ and replace vertex $z$ with $y_1$ in $G_2$ (see Figure \ref{chain-2}). Let 
$$D_1=\left(  D \cup \{y_1\}\right)-\left( V(G_2) \cup \{z\} \right).$$
We show that $D_1$ is a super dominating set for $G_1$. $u$ is  dominated by $y_1 \in D_1$ now and since $N(z)-\{u\}\subseteq D$, then $N(y_1)-\{u\}\subseteq D_1$. Hence $N(y_1)\cap \overline{D_1} = \{u\}$. The rest of vertices in $\overline{D_1}$ are dominated by the same vertex as before and the definition of the super dominating set holds. So $\gamma _{sp}(G_1)\leq|D_1|$. Now we consider to $G_2$. Let 
$$D_2=\left(  D \cup \{x_2\}\right)-\left( V(G_1) \cup \{z\} \right).$$
Since $x_2 \in D_2$, clearly all vertices in $\overline{D_2}$ are dominated by the same vertex as before. So the definition of the super dominating set holds and $\gamma _{sp}(G_2)\leq|D_2|$. By Proposition \ref{pro-disconnect}, super domination number of a disconnected graph with components $G_1$ and $G_2$ is the summation of cardinal of each super dominating set of them. Since 
$$D_1\cup D_2=\left( D \cup \{y_1,x_2\}\right) - \{z\},$$ 
and $D_1\cap D_2= \{\}$, then
$$  \gamma _{sp}(G_1)+\gamma _{sp}(G_2)\leq|D_1|+|D_2|=|D_1\cup D_2|=\gamma _{sp}(C(G_1,G_2))+1.$$

\item[(i.2)]
There does not exist $u\in N(z)$ such that  $N(z)\cap \overline{D} = \{u\}$. Same as previous case, we form $G_1$ and $G_2$. Let 
$$D_1=\left(  D \cup \{y_1\}\right)-\left( V(G_2) \cup \{z\} \right),$$
and
$$D_2=\left(  D \cup \{x_2\}\right)-\left( V(G_1) \cup \{z\} \right).$$
All vertices in $\overline{D_1}$ and $\overline{D_2}$ are dominated by the same vertex as before. So by similar argument as previous case, we have 
$$  \gamma _{sp}(G_1)+\gamma _{sp}(G_2)\leq \gamma _{sp}(C(G_1,G_2))+1.$$
\end{itemize}

\item[(ii)]
$z\notin D$. So there exists $v\in D$ such that $N(v)\cap \overline{D} = \{z\}$. We form $G_1$ and $G_2$ same as part (i.1). Without loss of generality, suppose that $v\in V(G_1)$. Let 
$$D_1= D - V(G_2) ,$$
and
$$D_2=\left(  D \cup \{x_2\}\right)- V(G_1).$$
$D_1$ is a super dominating set for $G_1$ because  $y_1$ is dominated by $v$  and  $N(v)\cap \overline{D} = \{y_1\}$, and the rest of vertices in $\overline{D_1}$ are dominated by the same vertex as before. So $\gamma _{sp}(G_1)\leq|D_1|$. Since $x_2 \in D_2$, all vertices in $\overline{D_2}$ are dominated by the same vertex as before  and the definition of super dominating set holds. So $D_2$ is a super dominating set for $G_2$. Hence $\gamma _{sp}(G_2)\leq|D_2|$. Since $D_1\cup D_2= D \cup \{x_2\}$, and $D_1\cap D_2= \{\}$, then
$$  \gamma _{sp}(G_1)+\gamma _{sp}(G_2)\leq\gamma _{sp}(C(G_1,G_2))+1.$$

\end{itemize}
Hence in all cases, $  \gamma _{sp}(G_1)+\gamma _{sp}(G_2)\leq\gamma _{sp}(C(G_1,G_2))+1$, and therefore we have the result.
\qed
	\end{proof}

	\begin{remark}
Bounds in the Theorem \ref{chain2-thm} are sharp. For the upper bound, it suffices to consider $G_1=G_2=P_3$. Then by Theorem \ref{thm-2}, $ \gamma _{sp}(G_1)=\gamma _{sp}(G_2)=2$. Now let $y_1$ and $x_2$ be the vertex with degree 2 in $G_1$ and $G_2$, respectively. One can easily check that $C(G_1,G_2)=K_{1,4}$, and by Theorem \ref{thm-2}, $\gamma _{sp}(C(G_1,G_2))=4=\gamma _{sp}(G_1)+\gamma _{sp}(G_2)$. For the lower bound, it suffices to consider $H_1=F_4$ and $H_2=F_5$, where $F_n$ is the friendship graph of order $n$. Then by Theorem \ref{Firend-thm}, $ \gamma _{sp}(H_1)=5$ and $\gamma _{sp}(H_2)=6$. Now let $y_1$  be the vertex with degree 8 in $H_1$ and $x_2$ be the vertex with degree 10 in $H_2$, respectively. One can easily check that $C(H_1,H_2)=F_9$, and by Theorem \ref{Firend-thm}, $\gamma _{sp}(C(H_1,H_2))=10=\gamma _{sp}(H_1)+\gamma _{sp}(H_2)-1$.
	\end{remark}

We end this section by an immediate result of 	Theorem \ref{chain2-thm}.
	
	\begin{corollary}
Let $G_1,G_2, \ldots , G_n$ be a finite sequence of pairwise disjoint connected graphs and let
		$x_i,y_i \in V(G_i)$. Let $C(G_1,...,G_n)$ be the chain of graphs $\{G_i\}_{i=1}^n$ with respect to the vertices $\{x_i, y_i\}_{i=1}^k$ which obtained by identifying the vertex $y_i$ with the vertex $x_{i+1}$ for $i=1,2,\ldots,n-1$ (Figure \ref{chain-n}). Then,
		$$ \left( \sum_{i=1}^{n} \gamma _{sp}(G_i) \right) - n \leq \gamma _{sp}(C(G_1,...,G_n)) \leq \sum_{i=1}^{n} \gamma _{sp}(G_i).$$
	\end{corollary}

\section{Super domination number of bouquet of graphs}
In this section, we consider to another special case of graphs which is obtained by point-attaching from primary subgraphs. 
Let $G_1,G_2, \ldots , G_n$ be a finite sequence of pairwise disjoint connected graphs and let
		$x_i \in V(G_i)$. Let $B(G_1,...,G_n)$ be the bouquet of graphs $\{G_i\}_{i=1}^n$ with respect to the vertices $\{x_i\}_{i=1}^n$ and obtained by identifying the vertex $x_i$ of the graph $G_i$ with $x$ (see Figure \ref{bouquet-n}).
		
			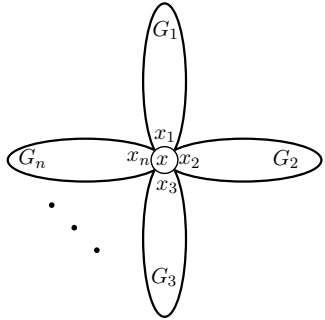
\begin{figure}[!h]
		\begin{center}
			\psscalebox{0.75 0.75}
			{
				\begin{pspicture}(0,-6.76)(5.6,-1.16)
				\rput[bl](2.6133332,-3.64){$x_1$}
				\rput[bl](3.0533333,-4.0933332){$x_2$}
				\rput[bl](2.6533334,-4.5466666){$x_3$}
				\rput[bl](2.5866666,-1.8133334){$G_1$}
				\rput[bl](4.72,-4.1066666){$G_2$}
				\rput[bl](2.56,-6.2){$G_3$}
				\rput[bl](2.1333334,-4.04){$x_n$}
				\rput[bl](0.21333334,-4.0933332){$G_n$}
				\psellipse[linecolor=black, linewidth=0.04, dimen=outer](1.4,-3.96)(1.4,0.4)
				\psellipse[linecolor=black, linewidth=0.04, dimen=outer](2.8,-2.56)(0.4,1.4)
				\psellipse[linecolor=black, linewidth=0.04, dimen=outer](4.2,-3.96)(1.4,0.4)
				\psellipse[linecolor=black, linewidth=0.04, dimen=outer](2.8,-5.36)(0.4,1.4)
				\psdots[linecolor=black, dotsize=0.1](0.8,-4.76)
				\psdots[linecolor=black, dotsize=0.1](1.2,-5.16)
				\psdots[linecolor=black, dotsize=0.1](1.6,-5.56)
				\psdots[linecolor=black, dotstyle=o, dotsize=0.5, fillcolor=white](2.8,-3.96)
				\rput[bl](2.6533334,-4.04){$x$}
				\end{pspicture}
			}
		\end{center}
		\caption{Bouquet of $n$ graphs $G_1,G_2, \ldots , G_n$ and $x_1=x_2=\ldots=x_n=x$.} \label{bouquet-n}
	\end{figure}

Clearly, bouquet of two graphs $G_1$ and $G_2$ with respect to vertices $x_1\in V(G_1)$ and $x_2\in V(G_2)$, is the same as chain of these two graphs. So by Theorem \ref{chain2-thm}, we have: 

	\begin{proposition}\label{bouquet2-prop}
	Let $G_1$ and $G_2$ be two  disjoint connected graphs and let
		$x_i \in V(G_i)$ for $i\in \{1,2\}$. Let $B(G_1,G_2)$ be the bouquet of graphs $\{G_i\}_{i=1}^2$ with respect to the vertices $\{x_i\}_{i=1}^2$ which obtained by identifying the vertex $x_1$ with the vertex $x_{2}$. Then,
$$\gamma _{sp}(G_1)+\gamma _{sp}(G_2) -1\leq   \gamma _{sp}(B(G_1,G_2)) \leq \gamma _{sp}(G_1)+\gamma _{sp}(G_2).$$
	\end{proposition}

Here we consider to bouquet of three graphs and find upper and lower bounds for the super domination number of that.

	\begin{theorem}\label{bouquet3-thm}
Let $G_1$, $G_2$ and $G_3$ be two  disjoint connected graphs and let
		$x_i \in V(G_i)$ for $i\in \{1,2,3\}$. Let $B(G_1,G_2,G_3)$ be the bouquet of graphs $\{G_i\}_{i=1}^3$ with respect to the vertices $\{x_i\}_{i=1}^3$ which obtained by identifying these vertices. Then,
$$\gamma _{sp}(G_1)+\gamma _{sp}(G_2)+\gamma _{sp}(G_3) -2\leq   \gamma _{sp}(B(G_1,G_2,G_3)) \leq \gamma _{sp}(G_1)+\gamma _{sp}(G_2)+\gamma _{sp}(G_3).$$	
	\end{theorem}

	\begin{proof}
First we consider to $G_1$ and $G_2$. Suppose that $H=B(G_1,G_2)$ with respect to the vertices $\{x_i\}_{i=1}^2$ which obtained by identifying the vertex $x_1$ with the vertex $x_{2}$. Let this vertex be $y$. Now we consider to graphs $H$ and $G_3$ and let $B(H,G_3)$  be the bouquet of these graphs with respect to the vertices $y$ and $x_3$. Clearly, we have $B(G_1,G_2,G_3)=B(H,G_3)$. First we find the lower bound. By Proposition \ref{bouquet2-prop}, we have: 
		\begin{align*}
		\gamma _{sp}(B(G_1,G_2,G_3)) &= \gamma _{sp}(B(H,G_3)) \\
		&\geq \gamma _{sp}(H)+\gamma _{sp}(G_3) -1 \\
		&= \gamma _{sp}(B(G_1,G_2))+\gamma _{sp}(G_3) -1 \\
		&\geq \gamma _{sp}(G_1)+\gamma _{sp}(G_2)+\gamma _{sp}(G_3) -2. 
		\end{align*}
By the same argument, we have $\gamma _{sp}(B(G_1,G_2,G_3)) \leq \gamma _{sp}(G_1)+\gamma _{sp}(G_2)+\gamma _{sp}(G_3)$, and therefore we have the result.		
\qed
	\end{proof}

As an immediate result of Proposition \ref{bouquet2-prop} and Theorem \ref{bouquet3-thm}, by using induction we have:

	\begin{corollary}\label{Cor-bouq-n}
Let $G_1,G_2, \ldots , G_n$ be a finite sequence of pairwise disjoint connected graphs and let
		$x_i,y_i \in V(G_i)$. Let $B(G_1,...,G_n)$ be the bouquet of graphs $\{G_i\}_{i=1}^n$ with respect to the vertices $\{x_i\}_{i=1}^k$ which obtained by identifying these vertices. Then,
		$$ \left( \sum_{i=1}^{n} \gamma _{sp}(G_i) \right) - n +1 \leq \gamma _{sp}(B(G_1,...,G_n)) \leq \sum_{i=1}^{n} \gamma _{sp}(G_i).$$
	\end{corollary}

We end this section by showing that bounds for $\gamma _{sp}(B(G_1,...,G_n))$ are sharp.

	\begin{remark}
Bounds in Corollary 	\ref{Cor-bouq-n} are sharp. For the lower bound, it suffices to consider $G_1=G_2= \ldots = G_n=F_2$ where $F_n$ is the friendship graph of order $n$ and  let $x_i$ for $i=1,2,\ldots,n$ be the vertex with degree 4 in $F_2$. One can easily check that $B(G_1,...,G_n)=F_{2n}$. By Theorem \ref{Firend-thm}, we have $\gamma _{sp}(B(G_1,...,G_n))=2n+1$. Also $\gamma _{sp}(F_2)=3$. So   $\gamma _{sp}(B(G_1,...,G_n))=\left( \sum_{i=1}^{n} \gamma _{sp}(G_i) \right) - n +1$. For the upper bound, it suffices to consider $H_1=H_2= \ldots = H_n=P_2$ where $P_n$ is the path graph of order $n$. Clearly, we have $\gamma _{sp}(P_2)=1$ and $B(H_1,...,H_n)=K_{1,n}$. By Theorem \ref{thm-2}, $\gamma _{sp}(B(H_1,...,H_n))=n$. Hence, $\gamma _{sp}(B(H_1,...,H_n))=\left( \sum_{i=1}^{n} \gamma _{sp}(H_i) \right)$.
	\end{remark}

\section{Conclusions}

In this paper, we obtained a sharp upper bound for super domination number of graphs which modified  by  operation $\odot$ on vertices.  Also
we presented some results for super domination number of chain and bouquet of finite sequence of pairwise disjoint connected graphs. Future topics of interest for future research include the following suggestions:

\begin{itemize}
\item[(i)]
Finding sharp lower bound for super domination number of $G\odot v$.
\item[(ii)]
Finding super  domination number of link and circuit of graphs.
\item[(iii)]
Finding super domination number of subdivision and power of a graph.
\item[(iv)]
Counting the number of super dominating sets of graph $G$ with size $k\geq \gamma _{sp}(G)$.
\end{itemize}

	\section{Acknowledgements} 

	The  author would like to thank the Research Council of Norway and Department of Informatics, University of Bergen for their support.

\end{document}